\documentclass[11pt,a4paper]{article}

\usepackage[margin=1.1in]{geometry}
\usepackage{amsmath,amssymb,amsthm,mathtools}
\usepackage{microtype}
\usepackage[hidelinks]{hyperref}
\usepackage{booktabs}
\usepackage{enumitem}
\usepackage{placeins}
\usepackage{needspace}

\newtheorem{theorem}{Theorem}[section]
\newtheorem{lemma}[theorem]{Lemma}
\newtheorem{proposition}[theorem]{Proposition}
\newtheorem{corollary}[theorem]{Corollary}
\theoremstyle{definition}
\newtheorem{definition}[theorem]{Definition}

\theoremstyle{remark}
\newtheorem{remark}[theorem]{Remark}

\newcommand{\wdim}{\operatorname{wdim}}

\newcommand{\Z}{\mathbb Z}
\newcommand{\ind}{\mathbf1}

\title{Weak \(k\)-metric dimension of Hamming graphs:\\
rectangular products and near-maximum parameters}
\author{Aryan Kumar\\[0.4em]
\small Department of Mathematics and Statistics\\
\small San Jos\'e State University\\
\small San Jos\'e, California, USA}
\date{}

\hypersetup{
  pdftitle={Weak k-metric dimension of Hamming graphs: rectangular products and near-maximum parameters},
  pdfauthor={Aryan Kumar}
}

\begin{document}
\maketitle

\begin{abstract}
We determine weak \(k\)-metric dimensions for several families of Hamming
graphs. For \(n\ge3\), \(m>n\), and \(3\le k\le2n\), a cyclic
construction proves Conjecture~6.1 of Fern\'andez, Klav\v{z}ar, Kuziak,
Mu\~noz-M\'arquez and Yero (2026) on \(K_n\square K_m\).
For these rectangular products, we also prove that
\(\wdim_2(K_n\square K_m)=m\) exactly when \(m\ge2n-2\).
For hypercubes \(Q_d\) with \(d\ge2\), we show that consecutive
parameters \(2s-1\) and \(2s\) have identical weak resolving sets.
Near the maximum parameter, a reduction to restricted-distance binary
codes determines \(\wdim_{2^d-t}(Q_d)\) for every feasible deficit
\(0\le t\le15\). For \(L\ge1\) and \(t\in\{2L,2L+1\}\), we
prove the stabilization formula
\(\wdim_{2^d-t}(Q_d)=2^d-L\) for \(d\ge2^L+1\), and show that
this threshold is sharp. For each fixed \(q\ge3\) and deficit \(t\),
we also determine \(\wdim_{2q^{d-1}-t}(K_q^{\square d})\) in all
sufficiently large dimensions, with an explicit sufficient condition.
\end{abstract}


\section{Introduction}

Let \(G\) be a finite connected graph with graph distance \(d_G\). For
\(S\subseteq V(G)\), put
\[
\Delta_S(x,y)=\sum_{w\in S}|d_G(x,w)-d_G(y,w)|.
\]
For a single landmark we abbreviate \(\Delta_{\{w\}}\) to \(\Delta_w\).
A set \(S\) is \emph{weak \(k\)-resolving} if
\(\Delta_S(x,y)\ge k\) for every pair of distinct vertices. Its minimum
cardinality is the \emph{weak \(k\)-metric dimension} \(\wdim_k(G)\).
This parameter was introduced by Peterin, Sedlar, \v{S}krekovski and
Yero in 2024~\cite{PeterinSedlarSkrekovskiYero2024}. At \(k=1\) it is the
classical metric dimension~\cite{Slater1975,HararyMelter1976}.
The \(k\)-metric dimension introduced by Estrada-Moreno,
Rodr\'iguez-Vel\'azquez and Yero~\cite{EstradaMorenoRodriguezVelazquezYero2015}
requires at least \(k\) landmarks to distinguish each pair. Weak
\(k\)-metric dimension instead sums the magnitudes of their distance
differences, so a single landmark can contribute more than one.
The largest feasible parameter is
\[
\kappa(G)=\min_{x\ne y}\Delta_{V(G)}(x,y).
\]
Each landmark contributes at most one to an adjacent pair, so
\(\wdim_k(G)\ge k\).

A \emph{Hamming graph} is a Cartesian product of complete graphs;
its distance counts the coordinates in which two vertices differ.
Fern\'andez, Klav\v{z}ar, Kuziak, Mu\~noz-M\'arquez and Yero (2026)
determined the weak \(k\)-metric dimension of \(K_n\square K_n\) and
proved that
\[
\kappa(K_{n_1}\square\cdots\square K_{n_d})
=2n_2\cdots n_d\qquad(n_1\ge\cdots\ge n_d\ge2)
\]
for \(d\ge2\)~\cite{FernandezKlavzarKuziakMunozYero2026}.
They conjectured an exact formula for rectangular products and asked
for results on higher-dimensional Hamming graphs. We address both
questions through coordinate counts. Farhan, Kuziak and
Yero~\cite{FarhanKuziakYero2026} have also studied weak \(k\)-metric
dimension for the direct product \(K_n\times K_n\), whose adjacency
relation differs from the Cartesian products considered here.

For rectangular products, an incidence matrix records the selected
landmarks. Bounds on its column sums give the conjectured lower bound;
a cyclic matrix attains it. More precisely, for \(n\ge3\), \(m>n\),
and \(3\le k\le2n\),
\[
\wdim_k(K_n\square K_m)=
\begin{cases}
m\lceil k/2\rceil,&k\text{ even},\\
m\lceil k/2\rceil-1,&k\text{ odd}.
\end{cases}
\]
For \(k=2\), we determine the exact equality threshold
\(m\ge2n-2\), refining the proposed sufficient range \(m\ge2n\).

In higher dimensions we instead record the omitted vertices. For the
hypercube \(Q_d=K_2^{\square d}\), write \(N=2^d\) and \(k=N-t\).
The columns of the omitted-vertex matrix form a binary code with
two-sided distance restrictions. Once pairs at distance at least three
have sufficient slack, these restrictions determine the weak dimension
exactly. This gives a partial-Hadamard characterisation, the sharp
stabilization theorem
\[
\wdim_{2^d-t}(Q_d)=2^d-\lfloor t/2\rfloor
\quad\text{for}\quad d\ge2^{\lfloor t/2\rfloor}+1,
\]
and explicit values through deficit \(15\). A parity argument shows
that consecutive parameters \(2s-1\) and \(2s\) have identical weak
resolving sets throughout the feasible range. The nonbinary case has a different local obstruction:
adjacent pairs bound the symbol frequencies in every coordinate.
A cyclic group action attains those bounds, proving eventually that
\[
\wdim_{2q^{d-1}-t}(K_q^{\square d})
=q^d-q\lfloor t/2\rfloor-(t\bmod2),\qquad q\ge3.
\]

C\'aceres et al.~\cite{CaceresHernandoMoraPelayoPuertasSearaWood2007}
studied ordinary metric dimension in Cartesian products. In hypercubes,
Kelen\v{c}, Masa Toshi, \v{S}krekovski and Yero~\cite{KelencToshiSkrekovskiYero2023}
compared metric, edge metric and mixed metric dimensions.
At \(k=1\), resolving hypercubes is also closely related to separating
systems and coin-weighing~\cite{ErdosRenyi1963,Lindstrom1964}.
Our near-maximum parameter regime leads to binary codes with prescribed
sets of distances. Such restrictions occur in the narrow-distance bounds
of Roth and Seroussi~\cite{RothSeroussi2007} and the few-distance
problems studied by Barg and Musin~\cite{BargMusin2011} and by Barg,
Glazyrin, Kao, Lai, Tseng and Yu~\cite{BargGlazyrinKaoLaiTsengYu2024}.
Here the permitted distances are specified by the deficit, and the code
columns must also distinguish the omitted vertices.

We use standard rank and coding bounds~\cite{MacWilliamsSloane1977},
the Singleton bound~\cite{Singleton1964}, Delsarte's linear programming
method~\cite{Delsarte1973}, and Hadamard
constructions~\cite{Sylvester1867,Paley1933,HedayatWallis1978}.
The additional capacity arguments needed for the exact tables are given
in full. Two capacities use finite exhaustive proofs, specified in the
appendix.
We do not determine the remaining rectangular \(k=2\) values below
the equality threshold, or the full feasible range of \(k\) in higher
dimensions.

\paragraph{Chronology.}
I publicly posted a complete proof of Conjecture~6.1 on
GitHub on 21 August 2026 at 07:00 UTC~\cite{Kumar2026Repository};
GitHub Actions records its processing at 07:01 UTC.
I subsequently became aware of a proof posted by
Patel on MathDB~\cite{MathDB2026Rectangular}, timestamped 22:05 UTC
that day, about fifteen hours after my public posting.
My sharp \(k=2\) equality theorem was also publicly posted on GitHub
at 19:35 UTC that day~\cite{Kumar2026K2}, approximately two and a half
hours before the MathDB posting, which does not contain that theorem.

\paragraph{Organization of the paper.}
Section~\ref{rect:sec:rectangular} proves Conjecture~6.1, establishes the
\(k=2\) equality threshold, and records its consequences for previously
reported values. Section~\ref{sec:binary} develops the hypercube reduction
and its connection to partial Hadamard matrices.
Section~\ref{sec:stabilization} proves parity pairing and sharp stabilization;
Section~\ref{sec:small} determines the small-deficit values, including
\(t=14,15\). Section~\ref{sec:nonbinary} treats nonbinary
Hamming graphs. Section~\ref{sec:future} discusses further coding
questions, and Appendix~\ref{sec:certs} specifies the two finite
capacity verifications.

\section{Rectangular Hamming graphs}\label{rect:sec:rectangular}

Write \(V(K_n\square K_m)=\mathbb Z_n\times\{0,\ldots,m-1\}\).
The first coordinate specifies a row and the second a column. For a
landmark \((a,b)\) and a vertex \((i,j)\),
\[
d((a,b),(i,j))=\mathbf1_{a\ne i}+\mathbf1_{b\ne j}.
\]
For \(m\ge n\ge2\), the largest feasible parameter is \(2n\).

\subsection{Coordinate identities}

To organize the distance calculations used in this section, represent
$S$ by its $n \times m$ binary incidence matrix $A$ and write
\begin{equation*}
  g_i = \sum_{j=0}^{m-1} A_{ij}, \qquad
  h_j = \sum_{i \in \Z_n} A_{ij}, \qquad
  |S| = \sum_i g_i = \sum_j h_j .
\end{equation*}
The following identities are those underlying Proposition~4.2 of
Fern\'andez et al.~\cite{FernandezKlavzarKuziakMunozYero2026}; we derive them directly
for completeness.

\medskip
\noindent\emph{Same-row pairs.} If $x=(i,j)$ and $y=(i,j')$ share a row, the
row indicators in $\Delta_w(x,y)$ cancel, and a landmark contributes $1$
exactly when its column is $j$ or $j'$:
\begin{equation}\label{eq:rect-1}
  \Delta_S(x,y) = h_j + h_{j'}.
\end{equation}

\noindent\emph{Same-column pairs.} Dually, if $x=(i,j)$ and $y=(i',j)$,
\begin{equation}\label{eq:rect-2}
  \Delta_S(x,y) = g_i + g_{i'}.
\end{equation}

\noindent\emph{Non-aligned pairs.} Let $x=(i,j)$, $y=(i',j')$ with $i \ne i'$,
$j \ne j'$. In the four layer sums $h_j + h_{j'} + g_i + g_{i'}$, each
selected endpoint is counted twice, which matches its contribution
$d(x,y)=2$. Every other landmark lying in exactly one of the two rows and two
columns is counted once, which also matches its contribution. The two cross
vertices $(i,j')$ and $(i',j)$ are counted twice but are equidistant from
$x$ and $y$. Subtracting those two excess contributions gives
\begin{equation}\label{eq:rect-3}
  \Delta_S(x,y) = h_j + h_{j'} + g_i + g_{i'} - 2A_{i,j'} - 2A_{i',j}.
\end{equation}

As observed by Fern\'andez et al.\ in
\cite[Proposition~4.1]{FernandezKlavzarKuziakMunozYero2026}, if $k\ge4$ and
all same-row and same-column pairs satisfy the weak $k$ inequalities, then
by \eqref{eq:rect-3} every
non-aligned pair satisfies
\begin{equation}\label{eq:rect-4}
  \Delta_S(x,y) \ge k + k - 2 - 2 = 2k - 4 \ge k.
\end{equation}
For $k=3$, the construction below supplies a stronger row-sum bound.

\subsection{Conjecture 6.1: the case
\texorpdfstring{$3 \le k \le 2n$}{3 <= k <= 2n}}
\label{rect:sec:conjecture}

The following theorem proves Conjecture~6.1 of Fern\'andez, Klav\v{z}ar,
Kuziak, Mu\~noz-M\'arquez and Yero~\cite{FernandezKlavzarKuziakMunozYero2026}.

\begin{theorem}\label{rect:thm:main}
For all $n \ge 3$, $m \ge n+1$, and $3 \le k \le 2n$,
\begin{equation*}
  \wdim_k(K_n \square K_m)=
  \begin{cases}
    m\lceil k/2\rceil, & k \text{ even},\\[1mm]
    m\lceil k/2\rceil - 1, & k \text{ odd}.
  \end{cases}
\end{equation*}
\end{theorem}

\subsubsection{Lower bound}

\begin{lemma}\label{rect:lem:main-lower}
Under the hypotheses of Theorem~\ref{rect:thm:main}, every weak $k$-resolving set
has at least $m\lceil k/2\rceil$ vertices when $k$ is even and at least
$m\lceil k/2\rceil-1$ vertices when $k$ is odd.
\end{lemma}

\begin{proof}

By \eqref{eq:rect-1}, every weak $k$-resolving set satisfies
\begin{equation}\label{eq:rect-5}
  h_j + h_{j'} \ge k \qquad (j \ne j').
\end{equation}
Let $a = \min_j h_j$.

\medskip
\noindent\emph{Even case, $k = 2r$.} If $a \ge r$ then
$|S| = \sum_j h_j \ge mr$. If $a \le r-1$, pairing a minimum column with
every other column in \eqref{eq:rect-5} gives
\begin{equation*}
  |S| \ge a + (m-1)(2r - a) = mr + (m-2)(r-a) \ge mr,
\end{equation*}
using $m \ge n+1 \ge 4$ and $r - a \ge 1$. Hence $|S| \ge mr$.

\medskip
\noindent\emph{Odd case, $k = 2r-1$.} If $a \ge r$ then $|S| \ge mr$, stronger
than needed. If $a \le r-1$, \eqref{eq:rect-5} gives
\begin{equation*}
  |S| \ge a + (m-1)(2r - 1 - a).
\end{equation*}
Increasing $a$ changes the right-hand side by $2-m$, which is negative because
$m\ge4$. Its minimum over integers $a \le r-1$ is therefore attained at
$a = r-1$:
\begin{equation*}
  |S| \ge (r-1) + (m-1)r = mr - 1.
\end{equation*}
In both parities every weak $k$-resolving set has at least the conjectured
number of vertices.
\end{proof}

\subsubsection{Cyclic construction}

\begin{lemma}\label{rect:lem:main-upper}
Under the hypotheses of Theorem~\ref{rect:thm:main}, there is a weak $k$-resolving
set of size $m\lceil k/2\rceil$ when $k$ is even and of size
$m\lceil k/2\rceil-1$ when $k$ is odd.
\end{lemma}

\begin{proof}

Put $r = \lceil k/2 \rceil$; the assumption $k \le 2n$ gives $r \le n$.
Define
\begin{equation}\label{eq:rect-6}
  S_0 = \{(i,j) \in \Z_n \times \{0,\dots,m-1\} :
          i - j \pmod n \in \{0,1,\dots,r-1\}\}.
\end{equation}
Each column of $S_0$ contains exactly $r$ vertices.  The rows
$j,j+1,\dots,j+r-1$ are distinct modulo $n$ because $r\le n$.  Among the
first $n$ columns, each row occurs exactly $r$ times: row $i$ occurs in the
columns $j\equiv i-t\pmod n$ for $t=0,\dots,r-1$. These are $r$ distinct
residues modulo $n$, each represented exactly once among columns
$0,\dots,n-1$. Since $m \ge n+1$, every row of $S_0$ contains at least $r$
vertices in total.

\medskip
\noindent\emph{Even case, $k = 2r$.} Take $S = S_0$. Every pair of column
sums equals $2r = k$ and every pair of row sums is at least $2r = k$. Even
$k$ in the stated range has $k \ge 4$, so \eqref{eq:rect-4} handles all non-aligned pairs.
Thus $S$ is weak $k$-resolving with $|S| = mr$.

\medskip
\noindent\emph{Odd case, $k = 2r-1\ge3$.} Here $r\ge2$. Column $n$ exists because
$m \ge n+1$, and it repeats the cyclic pattern of column $0$, since
$\{n+t \bmod n : 0 \le t < r\} = \{0,\dots,r-1\}$; in particular
$(0,n) \in S_0$. Delete it:
\begin{equation}\label{eq:rect-7}
  S = S_0 \setminus \{(0,n)\}.
\end{equation}
Column $n$ now has size $r-1$ and every other column has size $r$, so every
pair of columns has combined size at least $2r-1 = k$. Every row other than
row $0$ still has at least $r$ vertices, and row $0$ retains its $r$ vertices
among the first $n$ columns. The deleted vertex lies in column $n$, outside
those first $n$ columns. Thus every pair of row sums is at least $2r = k+1$.
For a non-aligned pair, \eqref{eq:rect-3} yields
\begin{equation*}
  \Delta_S(x,y) \ge (2r-1)+2r-4=4r-5\ge2r-1=k,
\end{equation*}
since $r\ge2$. This includes $k=3$. Thus $S$ is weak $k$-resolving with
$|S|=mr-1$.
\end{proof}

\begin{proof}[Proof of Theorem~\ref{rect:thm:main}]
Lemma~\ref{rect:lem:main-lower} gives the stated lower bound in each parity case,
and Lemma~\ref{rect:lem:main-upper} gives a weak $k$-resolving set attaining that
bound.  Therefore the two bounds are equal.
\end{proof}

\subsection{The case \texorpdfstring{$k=2$}{k=2}: the exact threshold}
\label{rect:sec:k2}

For $k=2$ the proposed formula reads $\wdim_2(K_n \square K_m) = m$.
Fern\'andez et al.~\cite{FernandezKlavzarKuziakMunozYero2026} propose
this equality for $m\ge2n$.
We determine the sharp threshold $2n-2$.

\begin{lemma}\label{rect:lem:lb}
Let $n \ge 2$ and $m \ge 2$. Every weak $2$-resolving set $S$ of
$K_n \square K_m$ satisfies $|S| \ge m$.
\end{lemma}

\begin{proof}
Fix any row $i$. For distinct columns $j \ne j'$ the pair $(i,j),(i,j')$ and
identity \eqref{eq:rect-1} force $h_j + h_{j'} \ge 2$. Let $a = \min_j h_j$. If $a \ge 1$
then $|S| = \sum_j h_j \ge m$. If $a = 0$, pick $j_0$ with $h_{j_0} = 0$;
then $h_j \ge 2$ for all $j \ne j_0$, so $|S| \ge 2(m-1) \ge m$ because
$m \ge 2$.
\end{proof}

\begin{lemma}\label{rect:lem:trans}
Let $m \ge 3$. If $S$ is weak $2$-resolving and $|S| = m$, then $h_j = 1$ for
every column $j$; that is, $S$ has exactly one vertex per column, so
$S = \{(f(j), j) : 0 \le j < m\}$ for a unique map
$f : \{0,\dots,m-1\} \to \Z_n$.
\end{lemma}

\begin{proof}
If some $h_{j_0} = 0$, the argument of Lemma~\ref{rect:lem:lb} gives
$|S| \ge 2(m-1) > m$, where the strict inequality follows from $m \ge 3$.
This is a contradiction. Hence
$h_j \ge 1$ for all $j$, and $\sum_j h_j = m$ forces $h_j = 1$ for every $j$.
\end{proof}

Such a set is called a \emph{transversal}; its row sums are
$g_i = |f^{-1}(i)|$.

\begin{lemma}\label{rect:lem:crit}
A transversal $S_f$ is weak $2$-resolving if and only if its row sums satisfy
\begin{align*}
  \text{\emph{(R1)}}\quad & g_i + g_{i'} \ge 2
    \quad\text{for all distinct rows } i,i';\\
  \text{\emph{(R2)}}\quad & g_i + g_{i'} \ge 4
    \quad\text{for all distinct rows } i,i' \text{ in the image of } f.
\end{align*}
\end{lemma}

\begin{proof}
We check the three pair types. Same-row pairs have
$\Delta = h_j + h_{j'} = 2$ identically. Same-column pairs give, by \eqref{eq:rect-2},
$\Delta((i,j),(i',j)) = g_i + g_{i'}$.  Every pair of distinct rows occurs
in each column, so these constraints are exactly (R1). For a non-aligned pair
$x=(i,j)$, $y=(i',j')$, identity \eqref{eq:rect-3} with $h \equiv 1$ reads
\begin{equation*}
  \Delta_S(x,y) = 2 + g_i + g_{i'}
    - 2\ind[f(j')=i] - 2\ind[f(j)=i'].
\end{equation*}
If both indicators equal $1$, then $i = f(j')$ and $i' = f(j)$ are distinct
rows in the image of $f$, and $\Delta = g_i + g_{i'} - 2$, which is $\ge 2$
iff $g_i + g_{i'} \ge 4$. Conversely, for any two distinct image rows $i,i'$,
choosing $j' \in f^{-1}(i)$ and $j \in f^{-1}(i')$ produces such a pair
$j \ne j'$ because $f(j) = i' \ne i = f(j')$.  Thus this family of
constraints is exactly (R2). If exactly one indicator equals $1$, then
$\Delta = g_i + g_{i'} \ge 2$ by (R1). If neither does, then
$\Delta = 2 + g_i + g_{i'} \ge 2$ by (R1). Since every pair of distinct
vertices is of exactly one type, (R1) and (R2) together are necessary and
sufficient.
\end{proof}

\begin{lemma}\label{rect:lem:constr}
Let $n \ge 3$ and $m \ge 2n-2$. Then a transversal satisfying \emph{(R1)} and
\emph{(R2)} exists; hence $\wdim_2(K_n \square K_m) \le m$.
\end{lemma}

\begin{proof}
Put $f(j)=1+(j\bmod(n-1))$ for $0\le j<m$. Row $0$ is empty, and every
other row has at least $\lfloor m/(n-1)\rfloor\ge2$ landmarks. A pair
containing row $0$ has combined size at least $2$, and any two occupied
rows have combined size at least $4$. Thus (R1) and (R2) hold, and
Lemma~\ref{rect:lem:crit} gives a weak $2$-resolving set of size $m$.
\end{proof}

\begin{remark}
The diagonal $f(j)=j\bmod n$ does not work at $m=2n-1$: its row sums are
$(2,\dots,2,1)$, and (R2) fails for a degree-$2$ row paired with the
degree-$1$ row. Concretely, at $(n,m) = (3,5)$ the pair $x=(0,2)$,
$y=(2,0)$ has $\Delta_S(x,y) = 1 < 2$, since both cross cells $(0,0)$ and
$(2,2)$ are landmarks. For $n\ge3$ and $m\ge3$, the same computation shows the diagonal is weak
$2$-resolving \emph{iff} $m \ge 2n$; the formula itself already starts at
$m=2n-2$ via the construction above.
\end{remark}

\begin{lemma}\label{rect:lem:sharp}
Let $n \ge 3$ and $3 \le m \le 2n-3$. Then
\begin{equation*}
  \wdim_2(K_n \square K_m) \ge m+1.
\end{equation*}
\end{lemma}

\begin{proof}
Suppose a weak $2$-resolving set $S$ with $|S| = m$ exists. Since $m \ge 3$,
Lemma~\ref{rect:lem:trans} makes $S$ a transversal, with row-sum vector $g$,
$\sum_i g_i = m \le 2n-3$, satisfying (R1) and (R2).

\emph{Case 1: some row is empty.} By (R1) there is exactly one empty row and
every other row has $g_i \ge 2$, so $\sum_i g_i \ge 2(n-1) = 2n-2 > m$, a
contradiction.

\emph{Case 2: no empty row.} Then the image of $f$ is all $n$ rows and (R2)
applies to every pair of distinct rows. Two distinct rows of weight $1$ would
violate (R2) since $1+1 = 2 < 4$, so at most one row has weight $1$. If one
row has weight $1$, (R2) forces every other row to have weight at least $3$,
whence $\sum_i g_i \ge 1 + 3(n-1) = 3n - 2 > 2n - 3 \ge m$, a
contradiction. Otherwise every row has weight at least $2$, so
$\sum_i g_i \ge 2n > m$, again a contradiction.

Hence no size-$m$ weak $2$-resolving set exists, and Lemma~\ref{rect:lem:lb} gives
$\wdim_2 \ge m+1$.
\end{proof}

\Needspace{8\baselineskip}
\begin{theorem}\label{rect:thm:k2}
For every $n \ge 3$:
\begin{enumerate}
  \item $\wdim_2(K_n \square K_m) = m$ for every $m \ge 2n-2$;
  \item $\wdim_2(K_n \square K_m) \ge m+1$ for every $3 \le m \le 2n-3$, so
        the formula $\wdim_2 = m$ fails there and the threshold $2n-2$ is
        sharp.
\end{enumerate}
\end{theorem}

\begin{proof}
(1) Lemma~\ref{rect:lem:constr} gives $\wdim_2 \le m$, and Lemma~\ref{rect:lem:lb}
gives $\wdim_2 \ge m$ because $m \ge 2n-2 \ge 4$. Part (2) is
Lemma~\ref{rect:lem:sharp}.
\end{proof}

\subsection{Consequences for previously reported values}
\label{rect:sec:comparison}

Table~1 of Fern\'andez et al.~\cite{FernandezKlavzarKuziakMunozYero2026} reports
$\wdim_2(K_5 \square K_8) = 9$ and $\wdim_2(K_5 \square K_9) = 10$, whereas
the true values are $8$ and $9$.  Lemma~\ref{rect:lem:constr} gives weak
$2$-resolving transversals with row-sum profiles $(0,2,2,2,2)$ and
$(0,3,2,2,2)$, respectively, and Lemma~\ref{rect:lem:lb} supplies the matching
lower bound.

Table~4 of the same paper reports $\wdim_2(K_8\square K_{15})=16$, whereas the true
value is $15$. Lemma~\ref{rect:lem:constr} gives a weak $2$-resolving transversal
with row-sum profile $(0,3,2,2,2,2,2,2)$, and Lemma~\ref{rect:lem:lb} proves
optimality.

The remark after their Conjecture~6.1 states that the equality
$\wdim_2=m$ fails for $m\in\{6,\dots,9\}$ when $n=5$, and for
$m\in\{7,\dots,11\}$ when $n=6$. Within $m\ge n+1$,
Theorem~\ref{rect:thm:k2} gives the failure sets $\{6,7\}$ and $\{7,8,9\}$,
respectively. Table~2 already gives the correct values $10$ and $11$ for
$K_6\square K_{10}$ and $K_6\square K_{11}$, contrary to that remark.

\section{Hypercubes and restricted-distance codes}\label{sec:binary}

Let \(Q_d\) be the graph on \(\{0,1\}^d\) with Hamming-distance
adjacency, and put \(N=2^d\). Throughout the hypercube sections,
\(d\ge2\), and \(d(x,y)\) denotes Hamming distance.

We study \(k=N-t\) with \(0\le t\le N-1\). Equivalently one specifies the omitted set \(T=V(Q_d)\setminus S\) of size \(q=\lvert T\rvert\). Then \(\Delta_S=\Delta_V-\Delta_T\). Integer entries in finite constructions denote binary expansions padded to the stated length.

\subsection{The full-cube sum}

We first compute the contribution obtained when every vertex is used as a
landmark.  By symmetry, this quantity depends only on the distance between the
two vertices.

\begin{lemma}\label{lem:walk}
For every integer \(r\ge 1\),
\[
\sum_{a=0}^r\binom{r}{a}\lvert r-2a\rvert = 2r\binom{r-1}{\lfloor(r-1)/2\rfloor}.
\]
\end{lemma}
\begin{proof}
The left-hand side is twice \(\sum_{a<r/2}(r-2a)\binom{r}{a}\). Using \(a\binom{r}{a}=r\binom{r-1}{a-1}\) and the symmetry of binomial coefficients, the inner sum collapses to \(r\binom{r-1}{\lfloor(r-1)/2\rfloor}\).
\end{proof}

\begin{theorem}\label{thm:fullcube}
If \(d(x,y)=r\), then
\[
F_d(r):=\Delta_{V(Q_d)}(x,y)=2^{d-r+1}\,r\binom{r-1}{\lfloor(r-1)/2\rfloor}.
\]
In particular \(F_d(1)=F_d(2)=N\), and \(F_d(r)\ge 3N/2\) for \(r\ge 3\).
\end{theorem}
\begin{proof}
Translate so that \(x=0\) and the first \(r\) coordinates of \(y\) are ones. A vertex with \(a\) ones among those \(r\) coordinates contributes \(\lvert r-2a\rvert\); the remaining \(d-r\) bits are free. Lemma~\ref{lem:walk} gives the formula. The evaluations at \(r=1,2,3\) are immediate, and \(F_d(r)/N\) is nondecreasing in consecutive odd/even pairs thereafter.
\end{proof}

\subsection{Universal lower bound and distance two}

Adjacent pairs give a universal lower bound.  Pairs at distance two then
translate the problem into constraints on the columns of a binary matrix.

\begin{proposition}\label{prop:adjacent}
For adjacent \(x,y\) one has \(\Delta_S(x,y)=\lvert S\rvert\). Consequently \(\wdim_{N-t}(Q_d)\ge N-t\), and every candidate complement satisfies \(q\le t\).
\end{proposition}
\begin{proof}
For adjacent vertices \(x,y\), every vertex of the hypercube is one unit
closer to one endpoint than to the other.  Each landmark therefore contributes
one to \(\Delta_S(x,y)\), so \(\Delta_S(x,y)=|S|\).  If \(S\) is weak
\((N-t)\)-resolving, then \(|S|\ge N-t\), or equivalently \(q\le t\).
\end{proof}

\begin{lemma}\label{lem:distance-two}
Let \(T=V(Q_d)\setminus S\), list its \(q\) vertices as the rows of a
binary matrix, and denote its columns by \(c_1,\dots,c_d\).  If \(S\) is
weak \((N-t)\)-resolving and \(L=\lfloor t/2\rfloor\), then
\begin{equation}\label{eq:d2}
q-L\le d_H(c_a,c_b)\le L\qquad(a\neq b).
\end{equation}
Hence \(q\le 2L\), and if \(q>L\) the columns are distinct.
\end{lemma}
\begin{proof}
Fix two coordinates \(a,b\).  The opposite vertices of a \(2\)-face receive
contribution two from omitted vertices whose entries in columns \(a,b\) have
one parity.  The other opposite pair receives contribution two from the
complementary parity class.  The two parity classes have sizes
\(d_H(c_a,c_b)\) and \(q-d_H(c_a,c_b)\). Each size must be at most \(L\), which
gives~\eqref{eq:d2}.  Adding its two inequalities gives \(q\le2L\).  If
\(q>L\), equal columns would contradict the lower inequality.
\end{proof}

\begin{definition}
For a set \(D\) of positive integers, let \(A(\ell,D)\) be the largest
size of a code in \(\{0,1\}^{\ell}\) whose distances between distinct
words lie in \(D\). An integer interval used as \(D\) includes its endpoints.
Let \(C_d(t)\) be the largest \(q\) for which a binary \(q\times d\) matrix exists with distinct rows and with every pair of columns satisfying~\eqref{eq:d2}. For \(q>L\), put \(\gamma(q,L)=A(q-1,[q-L,L])\). To normalize an omitted-vertex matrix, translate one row to zero and delete that row. Flipping a column replaces its distance from another column by \(q-d_H\), so the symmetric interval in~\eqref{eq:d2} is preserved. Set \(\gamma(q,L)=\infty\) for \(q\le L\).
\end{definition}

The definition and Lemma~\ref{lem:distance-two} give
\[
C_d(t)=\max\{q\le 2L:\gamma(q,L)\ge d\text{ and \(q\) distinct rows are possible}\}
\]
and the lower bound \(\wdim_{N-t}(Q_d)\ge N-C_d(t)\).

The matching construction uses the full-cube sum for pairs at distance at
least three.

\begin{lemma}\label{lem:upper-construction}
Suppose \(F_d(r)-rt\ge N-t\) for every \(3\le r\le d\).  If a complement
\(T\) is counted by \(C_d(t)\), then \(V(Q_d)\setminus T\) is weak
\((N-t)\)-resolving.
\end{lemma}
\begin{proof}
Pairs at distance one are safe by Proposition~\ref{prop:adjacent} because
\(|T|\le t\).  Lemma~\ref{lem:distance-two} gives the required bound for
pairs at distance two.  For a pair at distance \(r\ge3\), each omitted
vertex removes at most \(r\) by the triangle inequality.  Since
\(|T|\le C_d(t)\le2L\le t\), the remaining contribution is at least
\(F_d(r)-rt\ge N-t\).
\end{proof}

\begin{theorem}\label{thm:reduction}
Suppose \(F_d(r)-rt\ge N-t\) for every \(3\le r\le d\). Then
\[
\wdim_{N-t}(Q_d)=N-C_d(t).
\]
\end{theorem}
\begin{proof}
Lemma~\ref{lem:distance-two} gives the lower bound \(N-C_d(t)\).
Lemma~\ref{lem:upper-construction} gives a weak resolving set of that size.
\end{proof}

\begin{corollary}\label{cor:eventual-reduction}
For each fixed \(t\), the identity
\(\wdim_{2^d-t}(Q_d)=2^d-C_d(t)\) holds for all sufficiently large \(d\).
\end{corollary}
\begin{proof}
Theorem~\ref{thm:fullcube} gives \(F_d(r)\ge3N/2\) for \(r\ge3\).
Thus the hypothesis of Theorem~\ref{thm:reduction} follows once
\(N/2\ge td-t\), which holds for all sufficiently large \(d\).
\end{proof}

For the finite dimensions below, we check the hypothesis of
Theorem~\ref{thm:reduction} directly.  When it fails, pairs at larger
distances can force a smaller complement, so we return to the definition.

\subsection{Partial Hadamard matrices}

The extremal case of the lower bound has a standard interpretation in terms
of orthogonal sign vectors.

\begin{theorem}\label{thm:hadamard}
Let \(t>0\) be even and suppose Theorem~\ref{thm:reduction} applies. Then \(\wdim_{N-t}(Q_d)=N-t\) if and only if a \(t\times d\) sign matrix with pairwise orthogonal columns and distinct rows exists. In particular:
\begin{enumerate}[leftmargin=1.4em]
\item \(d\le t\), by linear independence;
\item if \(t\equiv 2\pmod{4}\) and \(d\ge 3\), equality is impossible: three pairwise distances would equal the odd number \(t/2\), but the sum of the three pairwise Hamming distances among any three binary vectors is even;
\end{enumerate}
\end{theorem}
\begin{proof}
Assume first that equality holds.  The omitted set has size \(q=t=2L\), so
Lemma~\ref{lem:distance-two} forces \(d_H(c_a,c_b)=L\) for every two
columns.  After mapping \(0\) to \(+1\) and \(1\) to \(-1\), the columns
are pairwise orthogonal sign vectors.  The rows remain distinct.  Conversely,
such a sign matrix gives a complement counted by \(C_d(t)\), and
Theorem~\ref{thm:reduction} gives equality.

Pairwise orthogonal columns are linearly independent, so \(d\le t\).  If
\(t\equiv2\pmod4\) and \(d\ge3\), three binary columns would have pairwise
distance \(t/2\), an odd number.  This is impossible because the sum of the
three pairwise Hamming distances among three binary vectors is even.
\end{proof}

For an odd deficit \(t\), separately, the distance-two bound gives
\(q\le2\lfloor t/2\rfloor=t-1\), so \(\wdim_{N-t}(Q_d)=N-t\) is impossible.

The relaxation satisfies \(C_d(2L)=C_d(2L+1)\), since both deficits
share the same load limit \(L\). The corresponding equality of weak
dimensions holds without any tail hypothesis, as we now prove.

\section{Paired deficits and sharp stabilization}\label{sec:stabilization}

The paired values follow from a parity property of the resolving sets
themselves.

\begin{theorem}[Parity pairing]\label{thm:parity-pairing}
Let \(d\ge2\) and \(1\le s\le2^{d-1}\). A subset of \(V(Q_d)\) is
weak \((2s-1)\)-resolving if and only if it is weak \(2s\)-resolving.
Consequently,
\[
\wdim_{2s-1}(Q_d)=\wdim_{2s}(Q_d).
\]
\end{theorem}
\begin{proof}
Suppose \(S\) is weak \((2s-1)\)-resolving. Choose two coordinates and
consider the two opposite pairs in a face on those coordinates. Each
landmark contributes two to one of these pairs and zero to the other.
Both sums are even and at least \(2s-1\), hence at least \(2s\).
Their sum is \(2|S|\), so \(|S|\ge2s\).

For a pair at odd distance, each landmark's distance difference is an
odd integer in absolute value, and is therefore at least one. This is
the bipartite parity observation of Peterin, Sedlar, \v{S}krekovski and
Yero~\cite[Lemma~12]{PeterinSedlarSkrekovskiYero2024}. Thus
\(\Delta_S(x,y)\ge|S|\ge2s\). For a pair at even distance, every
summand is even, so \(\Delta_S(x,y)\ge2s-1\) implies
\(\Delta_S(x,y)\ge2s\). The reverse implication is immediate.
\end{proof}

The matrix formulation then reduces the evaluation of these paired
values to restricted-distance codes.

\begin{theorem}\label{thm:general}
Let \(L\ge 1\) and \(t\in\{2L,2L+1\}\). Write \(\gamma(q,L)\) as above.
\begin{enumerate}[leftmargin=1.4em]
\item Always \(C_d(t)\le 2L\). If \(L\) is odd and \(d\ge 3\), then \(C_d(t)\le 2L-1\).
\item Suppose \(L\ge2\) and a Hadamard matrix of order \(2L\) exists. Then \(\gamma(2L,L)=\gamma(2L-1,L)=2L\). Moreover, if \(d_0\) is the least number of Hadamard columns that distinguish its \(2L\) rows, then \(C_d(t)=2L\) for \(d_0\le d\le2L\).
\item If \(q=2L-2\) and \(L\ge3\), parity extension turns the allowed distances \(L-2,L-1,L\) into \(\{L-2,L\}\) when \(L\) is even, and into \(\{L-1,L+1\}\) when \(L\) is odd. This permits the use of narrow-distance bounds such as those of Roth--Seroussi~\cite{RothSeroussi2007} when their hypotheses apply.
\item For \(q=L\), repeated columns are allowed and \(C_d(t)\ge\min(L,2^d)\). Thus \(C_d(t)=L\) for all large \(d\).
\end{enumerate}
If in addition \(F_d(r)-rt\ge N-t\) for \(r\ge3\), then \(\wdim_{N-t}(Q_d)=N-C_d(t)\). The inequality \(\gamma(q,L)\ge d\) is necessary for \(q>L\); it is sufficient only when the selected code columns also distinguish the \(q\) rows.
\end{theorem}

\begin{proof}
The first assertion is the distance-two bound and the three-column parity obstruction. A normalized order-\(2L\) Hadamard matrix gives an equidistant code of length \(2L-1\), distance \(L\), and size \(2L\), which is optimal by the equidistant rank bound. Puncturing gives \(2L\) words of length \(2L-2\) at distances \(L-1,L\). Since a Hadamard order greater than two is divisible by four, \(L\) is even; parity extension reverses the puncture and the same rank bound proves optimality. Choosing a least separating subset of Hadamard columns and then adding columns proves the claim about \(C_d(t)\). The parity-extension calculation in the third item is immediate. Finally, for \(q=L\) every pair of columns satisfies~\eqref{eq:d2}; choose \(L\) distinct rows when \(2^d\ge L\). Each \(q>L\) has finite code capacity, so no such \(q\) persists for all large \(d\). The final statements follow from Theorem~\ref{thm:reduction} and the definition of \(\gamma\).
\end{proof}

Theorem~\ref{thm:general} does not determine \(C_d(t)\) for every \(L\).
The remaining cases include instances of the Hadamard conjecture and open
restricted-distance coding problems. Section~\ref{sec:small} evaluates the
capacities needed here.

\begin{theorem}[Sharp stabilization]\label{thm:stabilization}
Let \(L\ge1\) and \(t\in\{2L,2L+1\}\). Then
\[
\wdim_{2^d-t}(Q_d)=2^d-L\qquad(d\ge2^L+1).
\]
At \(d=2^L\), one has \(\wdim_{2^d-t}(Q_d)\le2^d-L-1\).
Thus \(2^L+1\) is the least dimension from which the displayed formula
holds in every dimension.
\end{theorem}
\begin{proof}
If \(q>L\), a normalized column code has length \(q-1\) and minimum
distance at least \(q-L\). Puncturing any \(q-L-1\) positions is
injective, leaving a code of length \(L\). This is the Singleton bound~\cite{Singleton1964,MacWilliamsSloane1977},
and gives \(\gamma(q,L)\le2^L\). Hence \(d>2^L\) forces \(q\le L\).
Any \(L\) distinct omitted vertices satisfy the distance-two constraints.

At \(d=2^L\), instead take all binary words of length \(L\) as columns
and adjoin an all-zero row. These \(L+1\) rows are distinct, and column
distances lie between \(1\) and \(L\), as required.

It remains to check pairs at greater distance. For \(L\ge2\) and
\(d\ge2^L\),
\[
2^{d-1}\ge d(L+1)-2L.
\]
For \(L=2,d=4\) this is equality. For \(L\ge3,d=2^L\), it follows from
\(L+1\le d/2\) and \(2^d\ge d^2\); increasing \(d\) preserves the
inequality. Each construction omits at most \(L+1\) vertices, so for
\(r\ge3\) its remaining contribution is at least
\(3\cdot2^{d-1}-d(L+1)\ge2^d-2L\).
Adjacent pairs also meet this bound. For \(L=1,d=2\), omit two adjacent
vertices; for \(L=1,d\ge3\), omit one vertex and use
\(2^{d-1}\ge d-2\). These observations prove both assertions.
\end{proof}

\section{Exact values for small deficits}\label{sec:small}
\subsection{Deficits through thirteen}

We now evaluate the restricted-distance capacities required for
\(0\le t\le13\) and then verify the remaining small dimensions directly.
For \(Q_2\), the values at \(t=0,1,2,3\) are \(4,4,2,2\), respectively.
Indeed, distance two forces no omissions when \(t\le1\); two adjacent
landmarks are weak \(2\)-resolving, and a single landmark does not resolve
the cycle. We may therefore assume \(d\ge3\).

\begin{theorem}\label{thm:to13}
Let \(d\ge 3\), \(N=2^d\), and \(0\le t\le\min\{13,N-1\}\). Then \(\wdim_{N-t}(Q_d)\) equals
\[
\begin{cases}
N, & t=0,1,\\
N-1, & t=2,3,\\
N-4, & t=4,5\text{ and }3\le d\le 4,\\
N-2, & t=4,5\text{ and }d\ge 5,\\
N-5, & t=6,7\text{ and }3\le d\le 5,\\
N-4, & t=6,7\text{ and }6\le d\le 8,\\
N-3, & t=6,7\text{ and }d\ge 9,\\
N-8, & t=8,9\text{ and }4\le d\le 8,\\
N-6, & t=8,9\text{ and }9\le d\le 16,\\
N-4, & t=8,9\text{ and }d\ge 17,\\
N-9, & t=10,11\text{ and }4\le d\le 8,\\
N-7, & t=10,11\text{ and }9\le d\le 22,\\
N-6, & t=10,11\text{ and }23\le d\le 32,\\
N-5, & t=10,11\text{ and }d\ge 33,\\
N-12, & t=12,13\text{ and }4\le d\le 12,\\
N-10, & t=12,13\text{ and }13\le d\le 16,\\
N-8, & t=12,13\text{ and }17\le d\le 64,\\
N-6, & t=12,13\text{ and }d\ge 65.
\end{cases}
\]
\end{theorem}

The values for \(t=10,11\) use one finite computation,
\(\gamma(8,5)=8\), specified in Appendix~\ref{sec:certs}. The other
capacity bounds below have analytic proofs.

\subsubsection{Analytic capacities at deficit ten}

The nontrivial capacities for \(t=10,11\) are
\[
\gamma(9,5)=A(8,\{4,5\})=8,\qquad \gamma(8,5)=8,\qquad \gamma(7,5)=22.
\]
We first prove the length-eight bound.

\begin{theorem}\label{thm:t10q9}
A binary code of length \(8\) with all pairwise distances in \(\{4,5\}\) has at most \(8\) words. The bound is attained by \(\{0\}\) together with the seven weight-\(5\) vectors whose complementary triples form the Fano plane on seven of the eight coordinates.
\end{theorem}
\begin{proof}
Translate so that the code contains \(0\). Remaining weights lie in \(\{4,5\}\). Even-weight and odd-weight vectors have even mutual distance, so all distances inside each weight class equal \(4\). Mapping \(0,1\mapsto +1,-1\), one obtains two sets \(S_0\) (weight \(4\)) and \(S_1\) (weight \(5\)) of pairwise orthogonal vectors in \(\{\pm 1\}^8\), with \(\langle x,1\rangle=0\) for \(x\in S_0\), \(\langle y,1\rangle=-2\) for \(y\in S_1\), and \(\langle x,y\rangle=-2\) across the two sets.

Write \(a=\lvert S_0\rvert\) and \(b=\lvert S_1\rvert\). If \(b=0\), orthogonality to \(\mathbf1\) gives \(a\le7\). Otherwise let \(y'=y+\tfrac14\mathbf{1}\) for \(y\in S_1\). Each \(y'\) lies in \(\mathbf{1}^\perp\). The projection of every \(y'\) onto \(\mathrm{span}(S_0)\) is the same vector \(-\tfrac14\sum S_0\), so the residuals \(w_y\) live in a Euclidean space of dimension \(7-a\). Their Gram matrix is \(8I-\frac{a+1}{2}J\), with eigenvalues \(8\) (multiplicity \(b-1\)) and \(8-b(a+1)/2\). If the last eigenvalue is nonzero then \(b\le 7-a\). If it vanishes then \(b(a+1)=16\) and \(b-1\le 7-a\), hence \(a+b\le 8\); but \(b(a+1)=16\) has no integer solution with \(a+b=8\). Thus \(a+b\le 7\) in every case, so the original code has size at most \(8\).

The Fano construction attains \(8\): seven triples on \(\{0,\dots,6\}\), pairwise intersecting in one point, complementary to seven weight-\(5\) vectors of pairwise intersection \(3\) (distance \(4\)). For example, take the seven triples
\(013,045,125,234,026,146,356\), writing a triple by its entries.
\end{proof}

\begin{lemma}\label{lem:folded}
\(A(6,\{2,3,4,5\})=22\) and \(A(8,\{2,\ldots,7\})=93\).
\end{lemma}
\begin{proof}
For \(n=6,8\), form the Cayley graph on \(\mathbb F_2^n\) with
connection set \(\{e_1,\ldots,e_n,\mathbf1\}\). A permitted code is
an independent set. Its adjacency eigenvalues are
\(n-2j+(-1)^j\), with multiplicity \(\binom nj\).
For \(n=6\), the positive and negative multiplicities are \(22,42\);
for \(n=8\), they are \(163,93\). There are no zero eigenvalues.
The inertia bound, obtained by interlacing with the zero principal
submatrix of an independent set~\cite{Haemers1995}, gives an upper bound of \(22\) and \(93\), respectively.

Take all words of even weight below \(n/2\) and all words of odd weight
above \(n/2\). No two have distance one or \(n\). The sizes are
\(1+15+6=22\) for \(n=6\) and \(1+28+56+8=93\) for \(n=8\).
\end{proof}

\subsubsection{Proof of the endpoint table}\label{subsec:endpointproof}

The following lemma collects every normalized-code capacity needed for
Theorem~\ref{thm:to13}.  Each row lists all \(q\) with \(L<q\le2L\).

\begin{lemma}\label{lem:endpointcapacities}
\[
\begin{array}{c|l}
L&\text{exact capacities}\\ \hline
2&\gamma(4,2)=4,\quad\gamma(3,2)=4\\
3&\gamma(6,3)=2,\quad\gamma(5,3)=5,\quad\gamma(4,3)=8\\
4&\gamma(8,4)=8,\quad\gamma(7,4)=8,\quad
   \gamma(6,4)=\gamma(5,4)=16\\
5&\gamma(10,5)=2,\quad\gamma(9,5)=\gamma(8,5)=8,\quad
   \gamma(7,5)=22,\quad\gamma(6,5)=32\\
6&\gamma(12,6)=\gamma(11,6)=12,\quad
   \gamma(10,6)=\gamma(9,6)=16,\quad
   \gamma(8,6)=\gamma(7,6)=64.
\end{array}
\]
Only \(\gamma(8,5)=8\) is computer-assisted.
\end{lemma}
\begin{proof}
For \(L=2\), the first entry is the equidistant rank bound in length three,
and the second consists of all four length-two words.  For \(L=3\), three
binary words cannot have all three pairwise distances odd, proving
\(\gamma(6,3)=2\).  For \(q=5\), after translating to contain zero, the
nonzero supports have size two or three.  The size-two supports are a
pairwise-intersecting edge family on four points.  If there is one such edge,
at most two triples are compatible with it; with at least two edges, at most
one triple is compatible.  Hence there are at most four nonzero words, and a
five-word example attains the bound.  All eight length-three words give the
last entry.

For \(L=4\), the \(q=8\) entry is the equidistant rank bound, attained by an
order-eight Hadamard matrix.  For \(q=7\), parity extension gives a length-seven
equidistant code of distance four, again of size at most eight.  For \(q=6\),
the allowed words form an independent set of \(Q_5\), so the disjoint edges in
one coordinate direction give at most \(16\), attained by a parity class;
for \(q=5\), all \(16\) length-four words are allowed.

For \(L=5\), the odd-distance argument gives \(\gamma(10,5)=2\), and
Theorem~\ref{thm:t10q9} gives \(\gamma(9,5)=8\).  Appendix~\ref{sec:certs} gives the finite computation for
\(\gamma(8,5)=8\), and Lemma~\ref{lem:folded} gives \(\gamma(7,5)=22\).
All \(32\) length-five words give \(\gamma(6,5)=32\).

For \(L=6\), the first entry follows from the equidistant rank bound and an
order-12 Hadamard matrix.  Parity extension proves the same upper bound for
\(q=11\), with equality by puncturing the Hadamard code.  For \(q=10\), parity
extension changes the allowed distances to \(4,6\), and
\cite[Proposition~4.3(i)]{RothSeroussi2007} gives
\(4\cdot6\cdot10/(5^2-10)=16\); the linear \([9,4]\) code in
\cite[Example~4.2]{RothSeroussi2007} attains it.  Puncturing gives the
\(q=9\) value, with the matching bound from
\cite[Proposition~4.3(ii)]{RothSeroussi2007}.  Finally, an independent set of
\(Q_7\) has size at most \(64\), attained by a parity class, while all \(64\)
length-six words are allowed for \(q=7\).
\end{proof}

\begin{proof}[Proof of Theorem~\ref{thm:to13}]
The case \(t=0,1\) follows from Lemma~\ref{lem:distance-two}, since \(L=0\).  For
\(t=2L,2L+1\), Lemma~\ref{lem:endpointcapacities}, the three-column parity
obstruction, and the definition of \(C_d(t)\) give
\begin{equation}
\begin{array}{c|l}
L&C_d(2L)=C_d(2L+1)\\ \hline
1&1\\
2&4\ (3\le d\le4),\quad2\ (d\ge5)\\
3&5\ (3\le d\le5),\quad4\ (6\le d\le8),\quad3\ (d\ge9)\\
4&8\ (4\le d\le8),\quad6\ (9\le d\le16),\quad4\ (d\ge17)\\
5&9\ (4\le d\le8),\quad7\ (9\le d\le22),\quad
  6\ (23\le d\le32),\quad5\ (d\ge33)\\
6&12\ (4\le d\le12),\quad10\ (13\le d\le16),\quad
  8\ (17\le d\le64),\quad6\ (d\ge65).
\end{array}\label{eq:capacity-table}
\end{equation}
The constructions in Lemma~\ref{lem:row-separation} below supply a
separating prefix for each required capacity. Thus every entry of
\eqref{eq:capacity-table} satisfies the distinct-row condition.

Direct substitution in Theorem~\ref{thm:fullcube} shows that the tail
hypothesis of Theorem~\ref{thm:reduction} holds for even deficits
\(t=2,4,6,8,10,12\) beginning at dimensions \(3,5,6,6,6,7\), respectively.
For larger dimensions this propagates using \(F_d(r)\ge3N/2\); after the few
boundary cases, \(2^{d-1}\ge t(d-1)\) is sufficient and is preserved when
\(d\) increases.  The dimensions before those thresholds are checked
directly from the original definition by the following complements:
\[
\begin{array}{c|c|l|c}
t&d&T&\min_{x\ne y}\Delta_{V\setminus T}(x,y)\\ \hline
4&3&\{0,2,4,6\}&4\\
4&4&\{0,6,10,12\}&12\\
6&3&\{0,1,2,4,6\}&2\\
6&4&\{0,2,4,8,14\}&10\\
6&5&\{0,14,18,20,24\}&26\\
8&4&\{0,2,4,6,8,10,12,14\}&8\\
8&5&\{0,6,10,12,16,22,26,28\}&24\\
10&4&\{0,3,4,7,9,10,13,14,15\}&6\\
10&5&\{0,3,4,7,9,10,13,14,15\}&22\\
10&6&\{0,3,13,14,36,39,41,42,47\}&54\\
12&4&\{1,2,4,7,8,9,10,11,12,13,14,15\}&4\\
12&5&\{1,2,4,23,8,25,26,11,28,13,14,15\}&20\\
12&6&\{1,2,36,23,8,25,58,43,28,45,14,15\}&52.
\end{array}
\]
Each minimum equals \(N-t\).  The same complement meets the weaker requirement
for the odd partner \(2L+1\), and the distance-two upper bound is identical.
Combining \eqref{eq:capacity-table}, Theorem~\ref{thm:reduction}, and the displayed finite checks
gives every line of the theorem.
\end{proof}

\Needspace{6\baselineskip}
\begin{lemma}\label{lem:row-separation}
Every construction needed in \eqref{eq:capacity-table} can be ordered so
that the first applicable prefix distinguishes the omitted rows.
\end{lemma}
\begin{proof}
For the order-four and order-eight Hadamard arrays, use the binary arrays
\((u\cdot v)_{u,v\in\mathbb F_2^a}\), with \(a=2,3\), respectively.
Put the \(a\) unit-vector columns first. These distinguish all rows;
every pair of columns has distance \(2^{a-1}\).
For \(L=3,q=5\), use the normalized code \((3,5,9,0,14)\) of
length four. Its first three columns distinguish the five rows.

Whenever all words of length \(s\) are allowed, put the \(s\) unit
vectors first; these distinguish the \(s\) coordinate rows and the
adjoined zero row. For the even-parity code of odd length \(s\), put
\(e_i+e_s\), \(1\le i<s\), first. The resulting rows are the zero
row, the \(s-1\) unit rows and the all-one row, hence are distinct.
For the folded-cube construction of length six in Lemma~\ref{lem:folded},
the five columns \(e_i+e_6\) also belong to the code and distinguish
the seven rows. Its applicable range starts at \(d=9\).

For \(L=5,q=9\), the normalized code
\[
(244,206,217,186,0,227,173,151)
\]
has distances \(4,5\) and a separating four-column prefix. For
\(L=6,q=12\), use the following columns as length-twelve words,
without adjoining a row:
\begin{align*}
&(2729,3274,3852,4080,360,708,\\
&\qquad922,1212,1570,2064,2470,2686).
\end{align*}
Every pair has distance six, and its first four columns distinguish
the twelve rows. For \(L=6,q=10\), use the linear span of
\((15,60,325,401)\) in length nine, with these generators first.
Their coordinate patterns are distinct and nonzero, so adjoining the
zero row gives ten distinct rows. The nonzero weights of this code
are \(4,6\), as in \cite[Example~4.2]{RothSeroussi2007}.
Adding further columns preserves row separation. When \(q=L\), choose
any \(L\) distinct vertices instead.
\end{proof}

\subsection{Deficits fourteen and fifteen}

Here \(L=7\), so \(q\le14\). Parity forbids \(q=14\) for \(d\ge3\).

\begin{theorem}\label{thm:g13}
\(\gamma(13,7)=A(12,\{6,7\})=13\).
\end{theorem}
\begin{proof}
A construction of size \(13\) is the explicit code
\begin{align*}
&(3357,3749,884,1454,915,0,2537,\\
&\qquad 1752,3906,1139,591,2262,2618)
\end{align*}
in length \(12\).  Its \(78\) unordered pairs have distances \(6\) and \(7\), with
multiplicities \(42\) and \(36\). The first five columns distinguish all
thirteen rows after adjoining the zero row.

For the upper bound, translate to contain \(0\). Remaining weights are \(6\) or \(7\). The associated \(\{\pm 1\}^{12}\) vectors split into even-weight and odd-weight orthogonal sets \(S_0,S_1\), with cross inner products \(-2\), with \(S_0\perp\mathbf{1}\), and with \(\langle y,\mathbf{1}\rangle=-2\) for \(y\in S_1\). Put \(a=|S_0|\) and \(b=|S_1|\). If \(b=0\), the bound follows from orthogonality; assume \(b>0\). Eliminating the \(12I_a\) block from the Gram matrix of \(S_0\cup S_1\cup\{\mathbf{1}\}\) leaves \(12\) on the \((b-1)\)-dimensional subspace of \(S_1\)-coefficients summing to zero and the two-dimensional block
\[
\begin{pmatrix}12-ab/3&-2\sqrt b\\-2\sqrt b&12\end{pmatrix},
\]
whose determinant is \(4(36-b(a+1))\). Thus the Gram rank is \(a+b+1\), except when \(b(a+1)=36\), when it is \(a+b\). If \(a+b\ge13\), either rank is greater than \(12\), impossible for vectors in \(\mathbb{R}^{12}\). Hence at most \(12\) nonzero words, and the code has size at most \(13\).
\end{proof}

\begin{theorem}\label{thm:g12}
\(\gamma(12,7)=A(11,\{5,6,7\})=13\).
\end{theorem}
\begin{proof}
The length-\(11\) code
\[
(0,31,227,469,718,807,888,1142,1355,1452,1605,1721,1938)
\]
has the required distances and gives the lower bound.

Suppose a code with fourteen words exists. Append a parity coordinate and
map its words to sign vectors in \(\{\pm1\}^{12}\). Their distances are
\(6\) or \(8\), so their Gram matrix is \(G=12I-4A\), where \(A\)
is the adjacency matrix of the graph of distance-eight pairs. Such a pair
had distance seven before extension, so this graph is bipartite, with
parts given by the appended coordinate. Since \(G\) is positive
semidefinite of rank at most twelve, \(A\) has spectral radius at most
three and eigenvalue three with multiplicity at least two.

Consequently at least two connected components have spectral radius
three. In any such component, let \(u,v\) be the positive parts of a
Perron eigenvector. The corresponding null vector of \(G\) gives a zero
linear combination of the sign vectors. Reading the appended coordinate
shows \(\sum u=\sum v\). If the two part sizes are \(a,b\), summing
\(3u=Bv\) and \(3v=B^{\mathsf T}u\), where \(B\) is the bipartite
adjacency block, gives \(a,b\ge3\). If \(a=3\), equality forces every
vertex in the other part to have degree three. The component is
\(K_{3,b}\), whose spectral radius is \(\sqrt{3b}\), hence \(b=3\).
Thus each such component is either \(K_{3,3}\) or has at least eight
vertices. Two components on fourteen vertices force a \(K_{3,3}\).

Write its two sets of three sign vectors as \(x_1,x_2,x_3\) and
\(y_1,y_2,y_3\). The vectors within each part are orthogonal and
\(\langle x_i,y_j\rangle=-4\). Its null vector gives
\(\sum_i x_i+\sum_j y_j=0\). Flip coordinates so that the majority of
the three \(x\)-entries is positive in every coordinate. The identity
\(\|x_1+x_2+x_3\|^2=36\) shows that exactly three coordinates are
unanimous among the \(x_i\). On these coordinates all \(x_i\) are
positive and all \(y_j\) negative. On the other nine, exactly one
\(x_i\) is negative and exactly one \(y_j\) positive. Orthogonality
within each part makes each exceptional index occur three times; the
cross inner products then show that each ordered pair \((i,j)\) occurs
once. Index those nine coordinates by a \(3\times3\) array.

Every remaining sign vector is orthogonal to these six, since they form
a connected component. Let \(z\) be the sum of its first three entries,
and let \(Y\) be its remaining \(3\times3\) sign array. The six
orthogonality equations say that every row and column sum of \(Y\)
equals \(-z\). If \(z=\pm3\), there are just two possibilities, which
are antipodal, so at most one can occur. If \(z=1\), then
\(Y=2P-J\) for a permutation matrix \(P\); if \(z=-1\), then
\(Y=J-2P\). For a fixed \(P\), two distinct such vectors have distance
at most two when their arrays agree, and at least ten when their arrays
are opposite. Neither is allowed. The six permutation matrices therefore
contribute at most six vectors. There are at most seven further vectors
in total, contradicting the assumed fourteen words.
\end{proof}

\begin{theorem}\label{thm:g11}
\(\gamma(11,7)=A(11,\{4,6\})=17\).
\end{theorem}
\begin{proof}
Map a length-ten word \(x\) to \((x,0)\) if its weight is even, and
to \((\overline{x},1)\) otherwise. Words of the same parity retain
their distance; words of opposite parity have new distance
\(11-d_H(x,y)\). Thus distances \(4,5,6,7\) become \(4,6\).
Conversely, translate a length-eleven code with distances \(4,6\) to
contain zero. All words then have even weight, and the inverse map gives
distances in \(\{4,5,6,7\}\). This proves the first equality.
The length-ten code
\begin{align*}
&(0,15,227,117,188,342,218,422,626,\\
&\qquad695,724,750,795,808,845,897,505)
\end{align*}
has seventeen words, with unordered distance multiplicities
\(46,36,24,30\) at distances \(4,5,6,7\), respectively. Its first
fourteen columns distinguish the eleven rows after adjoining the zero row.

For the upper bound, suppose an eighteen-word length-eleven code exists.
Let \(e\) be its number of distance-four pairs, and let \(H\) be its
\(18\times11\) sign matrix. Its row Gram matrix gives
\[
\|HH^{\mathsf T}\|_F^2
=18\cdot11^2+2\bigl(9e+\tbinom{18}{2}-e\bigr)=2484+16e.
\]
Among eleven binary columns, at least
\(\binom52+\binom62=25\) pairs have the same weight parity. Each
corresponding sign inner product is congruent to two modulo four, so its
square is at least four. Hence
\[
\|H^{\mathsf T}H\|_F^2\ge11\cdot18^2+2\cdot25\cdot4=3764.
\]
The two Frobenius norms are equal, so \(e\ge80\). Some codeword has
at least nine distance-four neighbours. Translate that word to zero.
Among any six of its weight-four neighbours, the two-colouring by
distances four and six contains a monochromatic triangle. This is the
elementary Ramsey fact \(R(3,3)=6\).

Let the three supports in this triangle have common pairwise intersection
size \(a\in\{1,2\}\) and triple intersection size \(h\). Their Venn
regions have sizes \(h\), three copies of \(a-h\), three copies of
\(4-2a+h\), and \(3a-h-1\) outside the union. Thus, up to coordinate
permutation, there are exactly five possibilities, with
\((a,h)=(1,0),(1,1),(2,0),(2,1),(2,2)\). They are represented by the
four-word sets in Table~\ref{tab:five-configurations}.

Fourteen more words would have to be compatible with all four fixed
words and with one another. At least six of them must have weight four,
since zero has at least nine weight-four neighbours and only three have
been fixed. The finite verification in Appendix~\ref{sec:eleven-check}
rules out such an extension in all five cases. Therefore no eighteen-word
code exists, completing the upper bound.
\end{proof}

The Delsarte bound~\cite{Delsarte1973} used below follows from the polynomial
\[
K_j^{(n)}(x)=\sum_h(-1)^h\binom{x}{h}\binom{n-x}{j-h}.
\]

\begin{lemma}\label{lem:t14delsarte}
Let \(C\subseteq\{0,1\}^n\) have all nonzero distances in \(D\). If
\(g=\sum_{j=1}^n c_jK_j^{(n)}\), where \(c_j\ge0\) and \(g(a)\le-1\) for
every \(a\in D\), then \(\lvert C\rvert\le1+g(0)\). For \(n=9\) and \(D=\{3,4,5,6,7\}\), take
\(c_1=c_2=c_7=c_8=1/3\), \(c_9=1\), and all other coefficients zero.
Then \(1+g(0)=32\), proving \(\gamma(10,7)\le32\).
\end{lemma}
\begin{proof}
The Delsarte inequalities state that the average of every
\(K_j^{(n)}\) over ordered pairs of codewords is nonnegative. Multiplying by
the nonnegative \(c_j\), summing, and using \(g(a)\le-1\) off the diagonal
gives \(0\le |C|g(0)-|C|(|C|-1)\). The displayed values follow by direct
substitution in the integer formula for \(K_j^{(n)}\).
\end{proof}

\begin{theorem}\label{thm:t14}
Let \(d\ge4\), \(N=2^d\), and \(t\in\{14,15\}\). The exact values of
\(N-\wdim_{N-t}(Q_d)\) are given in Table~\ref{tab:t14}.
In particular, \(\wdim_{2^{18}-t}(Q_{18})=2^{18}-10\).
\end{theorem}
\begin{proof}
For \(d=4\), three landmarks cannot resolve the cube. Translate one
to zero, and call the others \(u,v\). If two coordinates have the same
pair \((u_i,v_i)\), their unit vectors are unresolved. Otherwise the
four coordinate pairs are \(00,01,10,11\); the vertices supported on
\(\{00,11\}\) and \(\{01,10\}\) then have the same distance to all
three landmarks. Thus \(\wdim_1(Q_4)\ge4\).
The four landmarks \(\{0,3,5,6\}\) have an order-four Hadamard sign
matrix \(H\). For distinct vertices, their distance differences form
\(Hz\) for a nonzero \(z\in\{0,\pm1\}^4\). Since
\(\|Hz\|_2^2=4\|z\|_2^2\), and \(\|Hz\|_1\) is an even integer,
its value is at least four when \(z\) has at least two nonzero entries;
with one entry it equals four directly. This proves the asserted value
for both deficits.

For \(5\le d\le 13\), Theorem~\ref{thm:g13} supplies \(13\) columns, and the first five already distinguish all thirteen rows. Direct evaluation of \(\Delta_{V\setminus T}\) on \(Q_5,Q_6,Q_7\) meets \(N-14\); for \(d\ge 7\) the tail inequality of Theorem~\ref{thm:reduction} holds for \(t=14\). The respective minima are \(18,50,114\), all equal to \(N-14\). The same construction meets the weaker \(N-15\) requirement, while both deficits have the same distance-two upper bound. Distance two forbids \(q=14\). Hence \(C_d(t)=13\).

For \(14\le d\le17\), Theorems~\ref{thm:g13} and~\ref{thm:g12}
give \(q\le11\). The seventeen columns in Theorem~\ref{thm:g11}
have a separating fourteen-column prefix. The tail inequality therefore
gives \(C_d(t)=11\) throughout this range.

For \(q=10\), the length-nine code
\begin{align*}
&(0,7,26,44,54,78,98,113,149,170,185,193,219,220,237,247,\\
&\qquad269,309,315,323,336,351,361,380,388,395,408,423,434,470,480,494)
\end{align*}
has all distances in \(\{3,\ldots,7\}\), and Lemma~\ref{lem:t14delsarte} proves \(\gamma(10,7)\le32\). Its first eighteen columns distinguish all ten row coordinates. Theorem~\ref{thm:g11} now yields \(C_d(t)=10\) for \(18\le d\le32\).

Lemma~\ref{lem:folded} gives \(\gamma(9,7)=93\).
The seven columns \(e_i+e_8\), \(1\le i<8\), distinguish the nine rows of that construction. For \(q=8\), all \(2^7=128\) normalized columns are allowed; the seven unit columns distinguish the eight rows. For \(q=L=7\), arbitrary repeated columns are allowed. Therefore \(C_d(t)=9\) for \(33\le d\le93\), \(C_d(t)=8\) for \(94\le d\le128\), and \(C_d(t)=7\) for \(d\ge129\). The tail inequality holds for these \(t\) from \(d=7\).
\end{proof}

\begin{table}[ht]
\centering
\caption{The number of omitted vertices
\(N-\wdim_{N-t}(Q_d)\), where \(N=2^d\) and \(t=14,15\).
At \(d=4\), the actual
value is \(12\), whereas the distance-two relaxation \(C_4(t)\) equals \(13\).}
\label{tab:t14}
\begin{tabular}{@{}ll@{}}
\toprule
\(d\) & \(N-\wdim_{N-t}(Q_d)\) \\
\midrule
\(4\) & \(12\) \\
\(5\)--\(13\) & \(13\) \\
\(14\)--\(17\) & \(11\) \\
\(18\)--\(32\) & \(10\) \\
\(33\)--\(93\) & \(9\) \\
\(94\)--\(128\) & \(8\) \\
\(\ge129\) & \(7\) \\
\bottomrule
\end{tabular}
\end{table}

To verify the distinction at \(d=4\), the thirteen rows
\(\{0,1,2,3,4,5,6,7,8,9,10,12,15\}\) have pairwise column
distances \(6\) or \(7\), so \(C_4(t)=13\); they cannot be the
complement of a weak resolving set of the required size.

\FloatBarrier
\section{Nonbinary Hamming graphs}\label{sec:nonbinary}

Let \(H(d,q)=K_q^{\square d}\), with \(N=q^d\). Distance is Hamming distance on \([q]^d\).

Fern\'andez, Klav\v{z}ar, Kuziak, Mu\~noz-M\'arquez and Yero
proved~\cite{FernandezKlavzarKuziakMunozYero2026} that the largest feasible
parameter is \(\kappa(H(d,q))=2q^{d-1}\).

The full-vertex sum needed below also has a direct formula.

\begin{proposition}\label{prop:hamF}
If \(d(x,y)=r\) in \(H(d,q)\), then
\[
F_{d,q}(r)=q^{d-r}\sum_{a+b+c=r}\binom{r}{a,b,c}(q-2)^c\,\lvert a-b\rvert,
\]
where \(a\) counts coordinates at which a landmark equals \(x\), \(b\)
counts those at which it equals \(y\), and \(c\) counts those at which it
equals neither.  In particular \(F_{d,q}(1)=2q^{d-1}=\kappa\), and
\(F_{d,2}(r)=F_d(r)\).
\end{proposition}
\begin{proof}
The coordinates on which \(x\) and \(y\) agree contribute the factor
\(q^{d-r}\).  Among the other \(r\) coordinates, suppose a landmark agrees
with \(x\) in \(a\) positions, with \(y\) in \(b\) positions, and with
neither in \(c\) positions.  There are
\(\binom{r}{a,b,c}(q-2)^c\) such patterns, and the absolute difference of
the two distances is \(|a-b|\).  Summing over \(a+b+c=r\) proves the
formula.  Substitution of \(r=1\) gives \(F_{d,q}(1)=2q^{d-1}\), and
setting \(q=2\) gives Theorem~\ref{thm:fullcube}.
\end{proof}

\begin{theorem}\label{thm:hamkappa}
\(\wdim_{\kappa}(H(d,q))=q^d\). Equivalently, no proper subset of \(V\) is weak \(\kappa\)-resolving.
\end{theorem}
\begin{proof}
For an adjacent pair differing in coordinate \(j\) with symbols \(a\neq b\), a landmark contributes \(1\) if and only if its \(j\)-th coordinate lies in \(\{a,b\}\). The full vertex set meets \(\kappa\) exactly. Omitting any vertex \(w\) with \(w_j\in\{a,b\}\) drops that pair strictly below \(\kappa\); choosing the pair so that this happens is always possible.
\end{proof}

For a deficit \(t\) from \(\kappa\), write \(k=\kappa-t\), let \(T=V\setminus S\), and put
\[
U_q(t):=q\left\lfloor\frac t2\right\rfloor+(t\bmod2).
\]

\begin{lemma}\label{lem:qarycapacity}
If \(S\) is weak \((\kappa-t)\)-resolving, then \(|T|\le U_q(t)\). This bound is attained by a single column-frequency profile: if \(t=2h\), use \((h,\ldots,h)\); if \(t=2h+1\), use \((h+1,h,\ldots,h)\).
\end{lemma}
\begin{proof}
For an adjacent pair differing in a fixed coordinate with symbols \(a\ne b\), the omitted contribution is \(f(a)+f(b)\), where \(f\) is the symbol-frequency vector in that column. Thus the two largest frequencies \(f_1\ge f_2\) satisfy \(f_1+f_2\le t\). If \(t=2h\), then \(f_2\le h\) and
\[
|T|\le f_1+(q-1)f_2\le2h+(q-2)h=qh.
\]
If \(t=2h+1\), the same argument gives \(|T|\le2h+1+(q-2)h=qh+1\). The displayed profiles attain the respective bounds.
\end{proof}

\begin{theorem}\label{thm:qaryeventual}
Fix \(q\ge3\) and \(t\ge0\). For every sufficiently large \(d\), with \(\kappa=2q^{d-1}>t\),
\[
\boxed{\wdim_{\kappa-t}(K_q^{\square d})=q^d-U_q(t).}
\]
More explicitly, the formula holds whenever \(d\ge2\), \(\kappa>t\), and
\begin{equation}
2(q-2)q^{d-2}\ge d\,U_q(t)-t. \label{eq:qary-threshold}
\end{equation}
\end{theorem}
\begin{proof}
Lemma~\ref{lem:qarycapacity} gives the lower bound. Identify the
alphabet with \(\mathbb Z/q\mathbb Z\). Simultaneously adding one to
all coordinates partitions the vertex set into \(q^{d-1}\) orbits of
size \(q\). Each orbit contains every symbol once in every coordinate.
Put \(h=\lfloor t/2\rfloor\) and take \(h\) complete orbits; when
\(t\) is odd, add one vertex from another orbit. Since \(t<2q^{d-1}\),
such an unused orbit exists. The resulting \(U_q(t)\) vertices are
distinct and each column has an attaining frequency profile. Thus every
adjacent pair loses at most \(t\).

For a pair at distance \(r\), divide Proposition~\ref{prop:hamF} by \(q^d\). The result is \(\mathbb E|Z_1+\cdots+Z_r|\), where independently \(Z_i=1,-1,0\) with probabilities \(1/q,1/q,(q-2)/q\). Conditional on a partial sum \(s\), adding one step changes the expected absolute value by \(2/q\) if \(s=0\) and by zero otherwise. Hence the full sums are nondecreasing in \(r\), and for every \(r\ge2\),
\[
F_{d,q}(r)\ge F_{d,q}(2)=\kappa+2(q-2)q^{d-2}.
\]
Each omitted vertex removes at most \(r\le d\). Under~\eqref{eq:qary-threshold}, every nonadjacent pair therefore retains at least \(\kappa-t\). The right side of~\eqref{eq:qary-threshold} grows only linearly in \(d\), while the left side grows exponentially, proving the eventual assertion.
\end{proof}

\section{Further questions}\label{sec:future}

The rectangular \(k=2\) values below the threshold \(m=2n-2\) remain
to be determined. For hypercubes, Theorem~\ref{thm:parity-pairing}
halves the number of distinct parameters, but does not determine the
values beyond the fixed-deficit range treated here. Their evaluation
requires further restricted-distance capacities and separating
constructions. The binary stabilization threshold in
Theorem~\ref{thm:stabilization} is sharp; determining a sharp threshold
for the nonbinary formula in Theorem~\ref{thm:qaryeventual} is another
natural question.

\appendix
\section{Finite capacity verification}\label{sec:certs}

Both verifications use a clique search bounded by greedy colourings,
a standard branch-and-bound method; see Tomita and
Seki~\cite{TomitaSeki2003}. We specify the graphs and the exhaustive
recursion used for the capacities in this paper.

\subsection{The seven-coordinate capacity}

One capacity in Theorem~\ref{thm:to13} uses exhaustive verification:
\(A(7,\{3,4,5\})=8\). Translate a code to contain zero. Form the
compatibility graph on the \(\binom73+\binom74+\binom75=91\) words
of weights \(3,4,5\), joining two words when their distance is in that
set. Its maximum clique has size seven. A clique search with a greedy
colouring bound proves the upper bound; a branch with current clique
size \(b\) and a \(c\)-colouring of its remaining candidates can be
discarded when \(b+c\le7\). The code
\[
(0,7,25,42,52,76,45,30)
\]
attains eight. The complete finite search is supplied with the source
in \path{anc/certify_appendix.py}.

\subsection{The eleven-coordinate capacity}\label{sec:eleven-check}

For each fixed set \(B\) in Table~\ref{tab:five-configurations}, form a
graph whose vertices are the words \(x\in\{0,1\}^{11}\) satisfying
\(d_H(x,b)\in\{4,6\}\) for every \(b\in B\). Join two vertices
when their distance is four or six. Let \(W\) be the vertices of weight
four. The required extension in Theorem~\ref{thm:g11} is a clique of
size fourteen containing at least six vertices of \(W\).

\begin{table}[ht]
\centering
\caption{The five configurations in Theorem~\ref{thm:g11}. None admits
an extension by fourteen compatible words with at least six of weight four.}
\label{tab:five-configurations}
\begin{tabular}{@{}cclc@{}}
\toprule
\(a\)&\(h\)& fixed words \(B\)&candidate vertices\\
\midrule
1&0&\(\{0,27,101,390\}\)&351\\
1&1&\(\{0,15,113,897\}\)&351\\
2&0&\(\{0,15,51,60\}\)&360\\
2&1&\(\{0,23,43,77\}\)&360\\
2&2&\(\{0,15,51,195\}\)&364\\
\bottomrule
\end{tabular}
\end{table}

The exhaustive check uses the following recursion. A state \((P,r,u)\)
asks whether the candidate set \(P\) contains an \(r\)-clique with at
least \(u\) vertices of \(W\). The initial state is \((V,14,6)\).
If \(r=0\), accept exactly when \(u=0\). Reject if \(|P|<r\) or
\(|P\cap W|<u\). Greedily partition \(P\) into independent sets and
list their vertices colour by colour as \(v_1,\ldots,v_m\). If
\(b_i\) is the colour of \(v_i\), the prefix through \(v_i\) has a
proper colouring with \(b_i\) colours. Process the list backwards.
Reject when \(b_i<r\); otherwise test
\[
\bigl(P\cap N(v_i),\ r-1,\ \max\{0,u-\mathbf1_{v_i\in W}\}\bigr),
\]
and, if it rejects, remove \(v_i\) from \(P\) and continue. Accept
if any recursive call accepts.

The colouring bound is valid because a clique contains at most one
vertex of each colour. Every clique either contains the current vertex
or is retained after its removal, so the recursion examines every
possible extension except branches excluded by proved upper bounds.
All five initial states reject. The complete integer-arithmetic
implementation is supplied as
\path{anc/certify_eleven_coordinate_capacity.cpp}; it requires only
a C++17 compiler and prints the result for each configuration. The
seventeen-word lower-bound construction is displayed in the theorem.

\FloatBarrier
\section*{Acknowledgments}

I am grateful to Dr.\ Jesse Geneson for his mentorship, support, and thoughtful
feedback throughout this work. His guidance helped me refine the arguments
and improve the presentation of the paper.

\section*{Declaration of generative AI and AI-assisted technologies}

The author directed the use of GPT-5.6 Sol and Kimi K3 in developing and
refining the proofs and manuscript through iterative discussion and review.
GPT-6 Astra assisted with further review of the arguments. The author takes
full responsibility for the content.

\begingroup
\small
\bibliographystyle{plain}
\bibliography{references/references}
\endgroup

\end{document}